\documentclass[a4paper,10pt, draft]{article}
\usepackage[english]{babel}
\usepackage{t1enc}
\usepackage{amsfonts}
\usepackage{amsmath,amsthm,amssymb,esint}
\usepackage{color}
\usepackage{authblk}
\usepackage{graphicx}
\usepackage{tikz}
\usetikzlibrary{calc,through,backgrounds}
\usepackage{mathrsfs} 
\usepackage{multirow}
\usetikzlibrary{arrows}

\newtheorem{thm}{Theorem}[section]
\newtheorem{coro}[thm]{Corollary}
\newtheorem{propo}[thm]{Proposition}
\newtheorem{lem}[thm]{Lemma}

\theoremstyle{remark}
\newtheorem{obs}[thm]{Remark}

\theoremstyle{definition}

\newcommand{\imply}{\ensuremath{\rightarrow}}

\newcommand{\yy}{\ensuremath{\wedge}}

\newcommand{\RR}{\ensuremath{\mathbb{R}}}

\newcommand{\menos}{\symbol{92}}

\newcommand{\embeds}{\ensuremath{\hookrightarrow}}

\allowdisplaybreaks 

\begin{document}

\title{Bessel potentials in Ahlfors regular metric spaces}

\vskip 0.3 truecm

\author{Miguel Andrés Marcos \thanks{The author was supported by Consejo Nacional
de Investigaciones Científicas y Técnicas, Agencia Nacional de Promoción Científica y Tecnológica and Universidad
Nacional del Litoral.\newline \indent Keywords and phrases:
Bessel potential, Ahlfors spaces, fractional derivative, Sobolev spaces
\newline \indent 2010 Mathematics Subject Classification: Primary
43A85.\newline }}
\affil{\footnotesize{Instituto de Matemática Aplicada del Litoral (CONICET-UNL)\\ Departamento de Matemática (FIQ-UNL)}}
\date{\vspace{-0.5cm}}

\maketitle

\begin{abstract}
In this paper we define Bessel potentials in Ahlfors regular spaces using a Coifman type approximation of the identity, and show they improve regularity for Lipschitz, Besov and Sobolev-type functions. We prove density and embedding results for the Sobolev potential spaces defined by them. Finally, via fractional derivatives, we find that for small orders, these Bessel potentials are inversible, and show a way to characterize potential spaces, using singular integrals techniques, such as the $T1$ theorem. Moreover, this characterization allows us to prove these spaces in fact coincide with the classical potential Sobolev spaces in the Euclidean case.
\end{abstract}

\setlength{\parskip}{10pt}

\section{Introduction}

Riesz and Bessel potentials of order $\alpha>0$ in $\RR^n$ are defined as the operators $\mathcal{I}_\alpha=(-\Delta)^{-\alpha/2}$ and $\mathcal{J}_\alpha=(I-\Delta)^{-\alpha/2}$ respectively, where $\Delta$ is the Laplacian and $I$ the identity. By means of the Fourier transform, it can be shown they are given by multipliers
\begin{align*}\left(\mathcal{I}_\alpha f\right)^\yy (\xi)= (2\pi|\xi|)^{-\alpha}\hat{f}(\xi), \hspace{1cm} \left(\mathcal{J}_\alpha f\right)^\yy(\xi)= (1+4\pi^2|\xi|^2)^{-\alpha/2}\hat{f}(\xi).\end{align*}

These frequency representations of Riesz and Bessel potentials, as well as of their associated fractional differential operators, depend on the existence of Fourier Transforms on the underlying space. In more general settings alternative tools are needed. Spaces such as self similar fractals are more general, but are still Ahlfors regular. In spaces with this type of regularity, scales are a good substitute of frequencies.

Both the Riesz potential and its inverse the fractional derivative $\mathscr{D}_\alpha=(-\Delta)^{\alpha/2}$, which on the frequency side is given by
\begin{align*}\left(\mathscr{D}_\alpha f\right)^\yy (\xi)= (2\pi|\xi|)^\alpha\hat{f}(\xi),\end{align*}
have an immediate generalization to metric measure spaces, as they take the form
\begin{align*}\mathcal{I}_\alpha f(x)=c_{\alpha,n}\int \frac{f(y)}{|x-y|^{n-\alpha}}dy, \hspace{1cm} \mathscr{D}_\alpha f(x)= \tilde{c}_{\alpha,n}\int\frac{f(y)-f(x)}{|x-y|^{n+\alpha}}dy,\end{align*}
at least for functions of certain integrability or regularity and $\alpha<2$. One can just replace $|x-y|^\alpha$ by a distance or quasi-distance  $d(x,y)^\alpha$, Lebesgue measure by a general measure and $|x-y|^n$ by the measure of the ball of center $x$ and radius $d(x,y)$.

For spaces of homogeneous type, fractional integrals (i.e. Riesz potentials) and derivatives, as well as their composition, have been widely studied. In the absence of Fourier transform, other techniques have been developed, such as the use of a Coifman type approximation of the identity (see for instance \cite{HS}, \cite{DJS}). It has been proven that even though the composition of a fractional integral and a fractional derivative (of the same order) is not necessarily the identity, at least for small orders of regularity it is an inversible singular integral. See \cite{GSV}, \cite{Ga} for the study of this composition in $L^2$ and \cite{HV} for Besov and Triebel-Lizorkin spaces.

Bessel potentials have essentially the same local behavior than Riesz potentials, but behave much better globally. For instance, they are bounded in every $L^p$ space, whereas $\mathcal{I}_\alpha$ is bounded from $L^p$ only to $L^q$ with $\frac{1}{p}-\frac{1}{q}=\frac{\alpha}{n}$. This leads to define potential spaces $\mathcal{L}^{\alpha,p}=\mathcal{J}_\alpha(L^p)$, and these coincide with Sobolev spaces when $\alpha$ is an integer. 

For $\alpha>0$, as
\begin{align*}2^{-\alpha/2}\leq\frac{1+(2\pi|\xi|)^\alpha}{(1+4\pi^2|\xi|^2)^{\alpha/2}}\leq 2,\end{align*}
the composition $(I+\mathscr{D}_\alpha)\mathcal{J}_\alpha$ is inversible in $L^2$. In fact, as shown in \cite{S}, for $1<p<\infty$ and $0<\alpha<2$,
\begin{align}\label{deriv}f\in \mathcal{L}^{\alpha,p} \text{ if and only if } f,\mathscr{D}_\alpha f\in L^p,\end{align}
and in terms of Riesz potentials,
\begin{align}\label{integ}f\in \mathcal{L}^{\alpha,p} \text{ if and only if } f\in L^p \text{ and there exists $\gamma\in L^p$ with }f=\mathcal{I}_\alpha\gamma.\end{align}

Bessel operators have been rarely studied in the metric setting, although in $\RR^n$ they can be represented as
\begin{align*}\mathcal{J}_\alpha f(x)=f*G_\alpha (x)=\int f(y)G_\alpha(x-y)dy,\end{align*}
where $G_\alpha$ is a radial function, so their definition does not present a limitation. In this paper we define Bessel-type potentials using the same construction found in \cite{GSV}. 

All the known tools and definitions used in this paper are described in section 2, such as approximations of the identity and singular integrals. In section 3 we define a Bessel-type potential operator and prove it increases the regularity of Lipschitz, Besov and Sobolev functions. In section 4 we describe the potential space obtained with this operator, and find relationships with Lipschitz, Besov and Sobolev functions, as well as a Sobolev embedding theorem. In section 5 we prove an inversion result for the Bessel operator using the techniques from \cite{GSV} and \cite{HV}. We finish this paper characterizing the potential space with the fractional derivative analogous to the Euclidean version in \ref{deriv} and with the fractional integral, analogous to \ref{integ}, and analyze the case of $\RR^n$.

\section{Preliminaries}

In this section we describe the geometric setting and basic results from harmonic analysis on spaces of homogeneous type needed to prove our results.

\subsection{The geometric setting}

We say $(X,\rho,m)$ is a space of homogeneous type if $\rho$ is a quasi-metric on $X$ and $m$ a measure such that balls and open sets are measurable and there exists a constant $C>0$ such that
\begin{align*}m_\rho(B(x,2r))\leq Cm(B_\rho(x,r))\end{align*}
for each $x\in X$ and $r>0$.

If $m(\{x\})=0$ for each $x\in X$, by \cite{MS} there exists a metric $d$ giving the same topology as $\rho$ and a number $N>0$ such that $(X,d,m)$ satisfies
\begin{align}\label{ahlf}m(B_d(x,2r))\sim r^N\end{align}
for each $x\in X$ and $0<r<m(X)$.

Spaces that satisfy condition \ref{ahlf} are called Ahlfors $N$-regular. Besides $\RR^n$ (with $N=n$), examples include self-similar fractals such as the Cantor ternary set or the Sierpi\'nski gasket.

Throughout this paper we will assume $(X,d,m)$ is Ahlfors $N$-regular. One useful property these spaces have is regarding the integrability of the distance function:
\begin{itemize}
	\item $\int_{B(x,r)}d(x,y)^sdm(y)<\infty$ if and only if $-N<s<\infty$, and here 
		\begin{align*}	\int_{B(x,r)}d(x,y)^sdm(y)\sim r^{s+N};\end{align*}
	\item $\int_{X\menos B(x,r)}d(x,y)^sdm(y)<\infty$ if and only if $-\infty<s<-N$, and here 
		\begin{align*}	\int_{X\menos B(x,r)}d(x,y)^sdm(y)\sim r^{s+N}.\end{align*}
\end{itemize}

If we add (locally integrable) functions we get
\begin{itemize}
	\item if $-N<s<\infty$, 
		\begin{align*}	\int_{B(x,r)}f(y)d(x,y)^sdm(y)\leq C r^{s+N}Mf(x);\end{align*}
	\item if $-\infty<s<-N$,
		\begin{align*}	\int_{X\menos B(x,r)}f(y)d(x,y)^sdm(y)\leq C r^{s+N}Mf(x),\end{align*}
\end{itemize}
where $Mf$ is the Hardy-Littlewood maximal function of $f$.

\subsection{Aproximations of the identity}

In Ahlfors spaces of infinite measure (and thus unbounded), Coifman-type aproximations of the identity can be constructed. In this paper we will work with a continuous version, as presented in \cite{GSV}. See \cite{HS} for the discrete version. The construction is as follows.

Let $(X,d,m)$ be an Ahlfors $N$-regular space with $m(X)=\infty$. Let $h:[0,\infty)\imply \RR$ be a non-negative decreasing $C^\infty$ function with $h\equiv 1$ in $[0,1/2]$ and $h\equiv 0$ in $[2,\infty)$. For $t>0$ and $f\in L^1_{loc}$, define
\begin{itemize}
	\item $T_tf(x)=\frac{1}{t^N}\int_X h\left(\frac{d(x,y)}{t}\right)f(y)dm(y)$;
	\item $M_tf(x)=\varphi(x,t)f(x)$, with $\varphi(x,t)=\frac{1}{T_t1(x)}$;
	\item $V_tf(x)=\psi(x,t)f(x)$, with $\psi(x,t)=\frac{1}{T_t\left(\frac{1}{T_t1}\right)(x)}$;
	\item $S_t f(x)=M_tT_tV_tT_tM_t f(x)=\int_X s(x,y,t)f(y)dm(y)$, where
	\begin{align*}s(x,y,t)=\frac{\varphi(x,t)\varphi(y,t)}{t^{2N}}\int_X h\left(\frac{d(x,z)}{t}\right)h\left(\frac{d(y,z)}{t}\right)\psi(z,t)dm(z). \end{align*}\end{itemize}

$(S_t)_{t>0}$ will be our aproximation of the identity, with kernel $s$. We now list some of the properties they possess, they can be found in \cite{GSV} for the case $N=1$.

\begin{enumerate}
	\item $S_t1\equiv1$ for all $t>0$;
	\item $s(x,y,t)=s(y,x,t)$ for $x,y\in X$, $t>0$;
	\item $s(x,y,t)\leq C/t^N$ for $x,y\in X$, $t>0$;
	\item $s(x,y,t)=0$ if $d(x,y)>4t$;
	\item $s(x,y,t)\geq C'/t^N$ if $d(x,y)<t/4$;
	\item $|s(x,y,t)-s(x',y,t)|\leq C''\frac{1}{t^{N+1}}d(x,x')$;
	\item $S_t$ is linear and continuous from $L^p$ to $L^p$;
	\item $S_tf\imply f$ pointwise when $t\imply 0$ if $f$ is continuous;
	\item $|S_tf(x)-f(x)|\leq Ct^\gamma$ for each $x$ if $f$ is Lipschitz-$\gamma$;
	\item $S_tf(x)\imply 0$ uniformly in $x$ when $t\imply\infty$ if $f\in L^1$;
	\item $s$ is continuously differentiable with respect to $t$.
\end{enumerate}

Continuity of a linear operator $T$ from $A$ to $B$ will be denoted throughout this paper as
\begin{align*}T:A\imply B.\end{align*}

To include an interesting example of an Ahlfors space satisfying $m(X)=\infty$ (and thus having a Coifman-type approximation of the identity), we can modify the Sierpi\'nski gasket $T$ by taking dilations (powers of 2): $\tilde{T}=\cup_{k\geq 1}2^kT$. This $\tilde{T}$ preserves some properties of the original triangle, including the Ahlfors character.

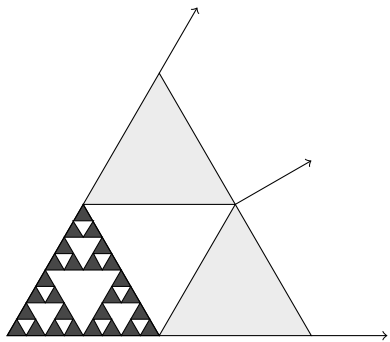
\begin{figure}[h!]\begin{center}\begin{tikzpicture}[scale=2]
\draw[fill=gray, fill opacity=0.15] (-0.5,0)--(1.5,0)--(0.5,1.74)--cycle;
\draw[fill=white] (0.5,0)--(0,0.87)--(1,0.87)--cycle;
\draw[<-] (2,0)--(1.5,0);
\draw[<-] (0.75,2.175)--(0.5,1.74);
\draw[<-] (1.5,1.16)--(1,0.87);
\draw[fill=black, fill opacity=0.7] (-0.5,0)--(0.5,0)--(0,0.87)--cycle;
\draw[fill=white] (0,0)--(0.25,0.435)--(-0.25,0.435)--cycle;
\draw[fill=white] (-0.25,0)--(-0.125,0.2175)--(-0.375,0.2175)--cycle;
\draw[fill=white] (0.25,0)--(0.375,0.2175)--(0.125,0.2175)--cycle;
\draw[fill=white] (0,0.435)--(0.125,0.6525)--(-0.125,0.6525)--cycle;
\draw[fill=white] (-0.375,0)--(-0.4375,0.10875)--(-0.3125,0.10875)--cycle;
\draw[fill=white] (0.375,0)--(0.4375,0.10875)--(0.3125,0.10875)--cycle;
\draw[fill=white] (-0.125,0)--(-0.1875,0.10875)--(-0.0675,0.10875)--cycle;
\draw[fill=white] (0.125,0)--(0.1875,0.10875)--(0.0675,0.10875)--cycle;
\draw[fill=white] (-0.25,0.2175)--(-0.3125,0.32625)--(-0.1875,0.32625)--cycle;
\draw[fill=white] (0.25,0.2175)--(0.3125,0.32625)--(0.1875,0.32625)--cycle;
\draw[fill=white] (-0.125,0.4350)--(-0.1875,0.54375)--(-0.0675,0.54375)--cycle;
\draw[fill=white] (0.125,0.4350)--(0.1875,0.54375)--(0.0675,0.54375)--cycle;
\draw[fill=white] (0,0.6525)--(-0.0675,0.76125)--(0.0675,0.76125)--cycle;
\end{tikzpicture}\end{center}\end{figure}

\subsection{Calderón reproducing formulas}

With this approximation of the identity $(S_t)_{t>0}$ we will construct our Bessel potential $J_\alpha$. For the proof relating $J_\alpha$ with the fractional derivative $D_\alpha$, we will follow the proof for the fractional integral as presented in \cite{GSV} and \cite{HV}, which requires the derivative of $S_t$ (that exists because $s$ is continuously differentiable with respect to $t$): let
\begin{align*}\frac{d}{dt}S_tf(x)=-\frac{1}{t}Q_tf(x),\end{align*} 
so
\begin{align*}Q_tf(x)=\int_X q(x,y,t)f(y)dm(y), \hspace{0.75cm}\text{with}\hspace{0.75cm}q(x,y,t)=-t\frac{d}{dt}s(x,y,t).\end{align*}

Some of their properties mirror those from $S_t$ and $s$:

\begin{enumerate}
	\item $Q_t1\equiv0$ for all $t>0$;
	\item $q(x,y,t)=q(y,x,t)$ for $x,y\in X$, $t>0$;
	\item $|q(x,y,t)|\leq C/t^N$ for $x,y\in X$, $t>0$;
	\item $q(x,y,t)=0$ if $d(x,y)>4t$;
	\item $|q(x,y,t)-q(x',y,t)|\leq C'\frac{1}{t^{N+1}}d(x,x')$;
	\item $Q_t: L^p\imply L^p$;
	\item \textsl{Calderón-type reproducing formulas.} (see \cite{C})
		\begin{align*}f=\int_0^\infty Q_tf\frac{dt}{t}, \hspace{1cm}f=\int_0^\infty\int_0^\infty Q_tQ_sf\frac{dt}{t}\frac{ds}{s}.\end{align*}
\end{enumerate}

\subsection{Singular Integrals}

In Ahlfors $N$-regular spaces, the following version of the $T1$ theorem hold (see for instance \cite{Ga}). Once again we require $m(X)=\infty$.

A continuous function $K:X\times X\menos\Delta\imply\RR$ (where $\Delta=\{(x,x):x\in X\}$) is a standard kernel if there exist constants $0<\eta\leq 1$, $C>0$ such that
\begin{itemize}
	\item $|K(x,y)|\leq Cd(x,y)^{-N}$;
	\item for $x\neq y$, $d(x,x')\leq cd(x,y)$ (with $c<1$) we have
	\begin{align*}|K(x,y)-K(x',y)|\leq Cd(x,x')^\eta d(x,y)^{-(N+\eta)};\end{align*}
	\item for $x\neq y$, $d(y,y')\leq cd(x,y)$ (with $c<1$) we have
	\begin{align*}|K(x,y)-K(x,y')|\leq Cd(y,y')^\eta d(x,y)^{-(N+\eta)}.\end{align*}
\end{itemize}

Let $C_c^\gamma$ denote the space of Lipschitz-$\gamma$ functions with compact support. A linear continuous operator $T:C_c^\gamma\imply (C_c^\gamma)'$ for $0<\gamma\leq 1$ is a singular integral operator with associated standard kernel $K$ if it satisfies
\begin{align*}\langle Tf,g\rangle = \iint K(x,y)f(y)g(x)dm(y)dm(x),\end{align*}
for $f,g\in C_c^\gamma$ with disjoint supports. If a singular integral operator can be extended to a bounded operator on $L^2$ it is called a Calderón-Zygmund operator or CZO.

Every CZO is bounded in $L^p$ for $1<p<\infty$, of weak type $(1,1)$, and bounded from $L^\infty$ to $BMO$.

The $T1$ theorem characterizes CZO's. We say that an operator is weakly bounded if
\begin{align*}|\langle Tf,g\rangle|\leq Cm(B)^{1+2\gamma/N}[f]_\gamma[g]_\gamma,\end{align*}
for $f,g\in C_c^\gamma(B)$, for each ball $B$.

\begin{thm}\label{tedeuno} ($\mathbf{T1}$) Let $T$ be a singular integral operator. Then $T$ is a CZO if and only if $T1,T^*1\in BMO$ and $T$ is weakly bounded.\end{thm}

\subsection{Besov spaces}

In metric measure spaces $(X,d,m)$, Besov spaces can be defined through a modulus of continuity, as seen in \cite{GKS}. For $1\leq p<\infty$ and $t>0$, the $p$-modulus of continuity of a locally integrable function $f$ is defined as
\begin{align*}E_pf(t)=\left(\int_X\fint_{B(x,t)}|f(x)-f(y)|^pdm(y)dm(x)\right)^{1/p},\end{align*}
where $\fint_A fdm$ denotes the average $\frac{1}{m(A)}\int_A fdm$, and the Besov space $B^\alpha_{p,q}$ for $\alpha>0$ and $1\leq q\leq \infty$ is the space of functions $f$ with the following finite norm
\begin{align*}\|f\|_{B^\alpha_{p,q}}=\|f\|_p+\left(\int_0^\infty t^{-\alpha q}E_pf(t)^q\frac{dt}{t}\right)^{1/q}\end{align*}
(with the usual modification for $q=\infty$).

For the case $p=q$, if the measure is doubling, an equivalent definition of the norm is
\begin{align*}\|f\|_{B^\alpha_{p,q}}=\|f\|_p+\left(\iint\frac{|f(x)-f(y)|^p}{d(x,y)^{\alpha p}m(B(x,d(x,y))}dm(y)dm(x)\right)^{1/q}.\end{align*}

\subsection{Sobolev spaces}

A way of defining Sobolev spaces in arbitrary metric measure spaces is Haj\symbol{170}asz approach (see \cite{H1} for the case $\beta=1$): a nonnegative function $g$ is a $\beta$-Haj\symbol{170}asz gradient of a function $f$ it the following inequality holds for almost every pair $x,y\in X$
\begin{align*}|f(x)-f(y)|\leq d(x,y)^\beta (g(x)+g(y)).\end{align*}

For $1\leq p\leq \infty$, the Haj\symbol{170}asz-Sobolev (fractional) space $M^{\beta,p}$ is defined as the space of functions $f\in L^p$ that have a gradient in $L^p$. Its norm is defined as
\begin{align*}\|f\|_{M^{\beta,p}}=\|f\|_p+\inf_g \|g\|_p\end{align*}
where the infimum is taken over all $\beta$-Haj\symbol{170}asz gradients of $f$.

For the case $p=\infty$, the space $M^{\beta,\infty}$ coincides with the space $C^\beta$ of bounded Lipschitz-$\beta$ functions.

Functions with $\beta$-Haj\symbol{170}asz gradients satisfy the following Poincaré inequality
\begin{align*}\fint_B |f-f_B|dm\leq C\text{diam}(B)^\beta\fint_B gdm,\end{align*}
for all balls $B$ (again, see \cite{H1} for the case $\beta=1$).

If the measure is doubling and $1\leq p<\infty$, then the following relationships hold between Besov and Sobolev spaces, for $\beta>0$ and $0<\epsilon<\beta$
\begin{align*}B^\beta_{p,p}\embeds M^{\beta,p}\embeds B^{\beta-\epsilon}_{p,p}\end{align*}
(see \cite{GKS}). Here the expression $A\embeds B$ means $A\subset B$ with continuous inclusion.

\section{Bessel potentials}

In this section we define the kernel $k_\alpha(x,y)$, to replace the convolution kernel $G_\alpha$ in the definition of $\mathcal{J}_\alpha$, and prove some properties this new Bessel-type potential operator $J_\alpha$ possesses, emulating those from $\mathcal{J}_\alpha$.

The convolution kernel $G_\alpha$ takes the form
\begin{align*}G_\alpha(x-y)=c_{n,\alpha}\int_0^\infty \left(t^\alpha e^{-t^2}\right)\left(t^{-n}e^{-\frac{1}{4}\left(\frac{|x-y|}{t}\right)^2}\right)\frac{dt}{t},\end{align*}
where $\varphi_t(x)=t^{-n}e^{-\frac{1}{4}\left(\frac{|x-y|}{t}\right)^2}$ is the Gaussian approximation of the identity. This provides us with a way to define the kernel in our context.

Let $(X,d,m)$ be our fixed Ahlfors $N$-regular space with $m(X)=\infty$, and $(S_t)_{t>0}$ an approximation of the identity as constructed in the previous section.

For $\alpha>0$, we define
\begin{align*}k_\alpha(x,y)=\alpha\int_0^\infty \frac{t^\alpha}{(1+t^\alpha)^2}s(x,y,t)\frac{dt}{t}.\end{align*}

Observe that the factor multiplying the approximation of the identity is $\frac{t^\alpha}{(1+t^\alpha)^2}$, as opposed to $t^\alpha e^{-t^2}$ in $G_\alpha$. It presents the same local behaviour, but near infinity it has only integrable decay. However, the properties obtained for $k_\alpha$ will be sufficient for our purposes.

The following properties follow immediately from definition and the properties of the kernel $s$, listed in section 2.

\begin{lem}\label{propk}$k_\alpha$ satisfies:
\begin{enumerate}
	\item $k_\alpha\geq 0$;
	\item $k_\alpha(x,y)=k_\alpha(y,x)$
	\item $k_\alpha(x,y)\leq C d(x,y)^{-(N-\alpha)}$;
	\item $k_\alpha(x,y)\leq C d(x,y)^{-(N+\alpha)}$ if $d(x,y)\geq4$;
	\item $|k_\alpha(x,z)-k_\alpha(y,z)|\leq C d(x,y)(d(x,z)\yy d(y,z))^{-(N+1-\alpha)}$;
	\item $|k_\alpha(x,z)-k_\alpha(y,z)|\leq C d(x,y)(d(x,z)\yy d(y,z))^{-(N+1+\alpha)}$ if $d(x,z)\geq 4$ and $d(y,z)\geq 4$;
	\item $\int_X k_\alpha(x,z)dm(z)=\int_X k_\alpha(z,y)dm(z)=1$ $\forall x,y$.
\end{enumerate}
\end{lem}

All results that will be presented in sections 3 and 4 involving the kernel $k_\alpha$ can be derived from just these properties. The actual need for the definition will become clear in section 5.
%

We are now able to define our Bessel potential
\begin{align*}J_\alpha g(x)=\int_X g(z)k_\alpha(x,z)dm(z).\end{align*}

Observe that from property $7$ of the last lemma, we get
\begin{align*}\|J_\alpha g\|_p\leq \|g\|_p\end{align*}
for $1\leq p\leq\infty$.

As expected, we can compare this operator with the Riesz potential, which is be defined from the kernel
\begin{align*}k'(x,y)=\int_0^\infty \alpha t^\alpha s(x,y,t)\frac{dt}{t}\sim \frac{1}{d(x,y)^{N-\alpha}}\end{align*}
as
\begin{align*}I_\alpha f(x)=\int_X f(y)k'(x,y)dm(y),\end{align*}
(see \cite{GSV}) and we obtain $|J_\alpha g(x)|\leq C I_\alpha |g|(x)$.

We now proceed to prove $J_\alpha$ improves regularity on Lipschitz, Besov and Haj\symbol{170}asz-Sobolev functions. We start with the Lipschitz case

\begin{propo}If $f=J_\alpha g$ and $\alpha+\beta<1$ for $\alpha,\beta>0$,
\begin{align*}|f(x)-f(y)|\leq C[g]_\beta d(x,y)^{\alpha+\beta}.\end{align*}
In particular, as $J_\alpha$ is bounded in $L^\infty$,
\begin{align*}J_\alpha: C^\beta\imply C^{\alpha+\beta}.\end{align*}
\end{propo}
\begin{proof}We will prove only the first part, the second follows immediately. What we will show also holds true for $I_\alpha$, as shown in \cite{GSV}. As $\int k_\alpha=1$, we have
\begin{align*}f(x)-f(y) &=\int_X g(z)\left(k_\alpha(x,z)-k_\alpha(y,z)\right)dm(z)\\
&=\int_X (g(z)-g(x))\left(k_\alpha(x,z)-k_\alpha(y,z)\right)dm(z),\end{align*}
and if we call $d=d(x,y)$
\begin{align*}|f(x)-f(y)|&\leq C\int_{B(x,2d)}\frac{|g(x)-g(z)|}{d(x,z)^{N-\alpha}}dm(z)\\
&\hspace{1cm}+C\int_{B(y,3d)}\frac{|g(x)-g(z)|}{d(y,z)^{N-\alpha}}dm(z)\\
&\hspace{1cm}+C\int_{X\menos B(x,2d)}|g(z)-g(x)|\left|k_\alpha(x,z)-k_\alpha(y,z)\right|dm(z)\\
&= I + II + III.\end{align*}
Then for $I$ and $II$, as $\alpha,\beta>0$,
\begin{align*}I\leq C[g]_\beta\int_{B(x,2d)}\frac{d(x,z)^\beta}{d(x,z)^{N-\alpha}}dm(z)\leq C[g]_\beta d^{\alpha+\beta},\end{align*}
\begin{align*}II\leq C[g]_\beta d^\beta\int_{B(y,3d)}\frac{1}{d(y,z)^{N-\alpha}}dm(z)\leq C[g]_\beta d^{\alpha+\beta}.\end{align*}
Finally, as $d(x,z)\sim d(y,z)$ for $z\in X\menos B(x,2d)$, and as $\alpha+\beta<1$,
\begin{align*}III\leq C[g]_\beta d\int_{X\menos B(x,2d)}d(x,z)^\beta d(x,z)^{-(N+1-\alpha)}dm(z)\leq C[g]_\beta d^{\alpha+\beta}.\end{align*}
\end{proof}

Before proving the increase in Besov regularity, we need the following lemma, that follows from properties 3 and 5 of \ref{propk}:

\begin{lem}\label{ineqk}For $q>0$ and $x,y\in X$,
\begin{itemize}
	\item if $q(N-\alpha)<N$,
	\begin{align*}\int_{d(x,z)<2d(x,y)}|k_\alpha(x,z)-k_\alpha(y,z)|^qdm(z)\leq C d(x,y)^{N-q(N-\alpha)};\end{align*}
	\item if $N<q(N-\alpha+1)$,
	\begin{align*}\int_{d(x,z)\geq2d(x,y)}|k_\alpha(x,z)-k_\alpha(y,z)|^qdm(z)\leq C d(x,y)^{N-q(N-\alpha)}.\end{align*}
\end{itemize}\end{lem}

\begin{propo}If $f=J_\alpha g$ and $\alpha+\beta<1$ for $\alpha,\beta>0$,
\begin{align*}\iint_{X\times X}\frac{|f(x)-f(y)|^p}{d(x,y)^{N+(\alpha+\beta)p}}dm(y)dm(x)\leq C\iint_{X\times X}\frac{|g(x)-g(z)|^p}{d(x,z)^{N+\beta p}}dm(z)dm(x).\end{align*}
In particular, as $J_\alpha$ is bounded in $L^p$,
\begin{align*}J_\alpha: B^\beta_{p,p}\imply B^{\alpha+\beta}_{p,p}.\end{align*}\end{propo}
\begin{proof}
Using $\int k_\alpha=1$, by Hölder's inequality we have
\begin{align*}|f(x)-f(y)|^p &\leq\\
&\hspace{-1cm}\leq C\left(\int_{B(x,2d(x,y))} |g(x)-g(z)|^p|k_\alpha(x,z)-k_\alpha(y,z)|dm(z)\right)\\
&\hspace{0cm}\times\left(\int_{B(x,2d(x,y))} |k_\alpha(x,z)-k_\alpha(y,z)|dm(z)\right)^{p/p'}\\
&\hspace{-0.5cm}+C\left(\int_{B(x,2d(x,y))^c} |g(x)-g(z)|^p|k_\alpha(x,z)-k_\alpha(y,z)|^{\theta p}dm(z)\right)\\
&\hspace{0cm}\times\left(\int_{B(x,2d(x,y))^c} |k_\alpha(x,z)-k_\alpha(y,z)|^{(1-\theta)p'}dm(z)\right)^{p/p'}.\end{align*}
By the previous lemma, if we find $0\leq \theta\leq 1$ such that $N<(1-\theta)p'(N-\alpha+1)$, we get
\begin{align*}|f(x)-f(y)|^p &\leq\\
&\hspace{-1cm}\leq C d(x,y)^{p\alpha-\alpha}\int_{B(x,2d(x,y))} |g(x)-g(z)|^p|k_\alpha(x,z)-k_\alpha(y,z)|dm(z)\\
&\hspace{-0.5cm}+C d(x,y)^{-N+p\alpha+\theta p(N-\alpha)}\\
&\hspace{0cm}\times\int_{B(x,2d(x,y))^c} |g(x)-g(z)|^p|k_\alpha(x,z)-k_\alpha(y,z)|^{\theta p}dm(z).\end{align*}
With this, to conclude the theorem it will be enough to prove
\begin{align*}\int_{d(x,z)<2d(x,y)} \frac{|k_\alpha(x,z)-k_\alpha(y,z)|}{d(x,y)^{N+\beta p+\alpha}}dm(y)\leq C \frac{1}{d(x,z)^{N+\beta p}}.\end{align*}
and for the other part
\begin{align*}\int_{d(x,z)\geq 2d(x,y)} \frac{|k_\alpha(x,z)-k_\alpha(y,z)|^{\theta p}}{d(x,y)^{2N+\beta p-\theta p(N-\alpha)}}dm(y)\leq C \frac{1}{d(x,z)^{N+\beta p}}.\end{align*}
\begin{itemize}
	\item For the first one, if $d(x,z)<2d(x,y)$ then $d(y,z)<3d(x,y)$ and by using the bound for $k_\alpha$,
	\begin{align*}\hspace{-1cm}\int_{d(x,z)<2d(x,y)} \frac{|k_\alpha(x,z)-k_\alpha(y,z)|}{d(x,y)^{N+\beta p+\alpha}}dm(y) &\leq\\
	&\hspace{-6cm}\leq C\int_{d(x,z)<2d(x,y)} \frac{1}{d(x,y)^{N+\beta p+\alpha}}\left(\frac{1}{d(x,z)^{N-\alpha}}+\frac{1}{d(y,z)^{N-\alpha}}\right)dm(y),\end{align*}
	then we consider two cases,
		\begin{itemize}
			\item if $d(y,z)<\frac{3}{2}d(x,z)<3d(x,y)$, then
			\begin{align*}\hspace{-2cm}\int_{d(y,z)<\frac{3}{2}d(x,z)<3d(x,y)}\frac{1}{d(x,y)^{N+\beta p+\alpha}}\left(\frac{1}{d(x,z)^{N-\alpha}}+\frac{1}{d(y,z)^{N-\alpha}}\right)dm(y) &\leq\\
			&\hspace{-8cm}\leq C\frac{1}{d(x,z)^{N+\beta p+\alpha}} \int_{d(y,z)<\frac{3}{2}d(x,z)}\frac{1}{d(y,z)^{N-\alpha}}dm(y)\\
			&\hspace{-8cm}\leq C\frac{1}{d(x,z)^{N+\beta p}};\end{align*}
			
			\item if $\frac{3}{2}d(x,z)\leq d(y,z)<3d(x,y)$,
			\begin{align*}\hspace{-2cm}\int_{\frac{3}{2}d(x,z)\leq d(y,z)<3d(x,y)}\frac{1}{d(x,y)^{N+\beta p+\alpha}}\left(\frac{1}{d(x,z)^{N-\alpha}}+\frac{1}{d(y,z)^{N-\alpha}}\right)dm(y) &\leq\\
			&\hspace{-8cm}\leq C\frac{1}{d(x,z)^{N-\alpha}}\int_{d(x,y)>d(x,z)/2}\frac{1}{d(x,y)^{N+\beta p+\alpha}}dm(y)\\
			&\hspace{-8cm}\leq C\frac{1}{d(x,z)^{N+\beta p}}.\end{align*}
		\end{itemize}
	
	\item For the second one, if $d(x,z)\geq 2d(x,y)$, then $d(x,z)\sim d(y,z)$ and by property 5 in \ref{propk},
	\begin{align*}\int_{d(x,z)\geq2d(x,y)} \frac{|k_\alpha(x,z)-k_\alpha(y,z)|^{\theta p}}{d(x,y)^{2N+\beta p-\theta p(N-\alpha)}}dm(y) &\leq\\
	&\hspace{-6cm}\leq C\frac{1}{d(x,z)^{\theta p(N-\alpha+1)}}\int_{d(x,z)\geq2d(x,y)} \frac{d(x,y)^{\theta p}}{d(x,y)^{2N+\beta p-\theta p(N-\alpha)}}dm(y)\\
	&\hspace{-6cm}\leq C\frac{1}{d(x,z)^{N+\beta p}}\end{align*}
	as long as $N+\beta p<\theta p(N-\alpha+1)$.
\end{itemize}

Finally, both conditions over $\theta$ can be rewritten as
\begin{align*}N+\beta p<\theta p(N-\alpha+1)<N+(1-\alpha)p,\end{align*}
and there is always a value for $\theta$ satisfying them, for $\beta<1-\alpha$.
\end{proof}

We have now the following result regarding Sobolev regularity.

\begin{propo}Let $f,g$ satisfy, for a.e. $x,y$,
\begin{align*}|f(x)-f(y)|\leq d(x,y)^\beta(g(x)+g(y)),\end{align*}
with $g\geq0$, $\beta>0$. Then for $\alpha>0$ and $\alpha+\beta<1$,
\begin{align*}|J_\alpha f(x)-J_\alpha f(y)|\leq Cd(x,y)^{\alpha+\beta}(Mg(x)+Mg(y)).\end{align*}
In particular, if $p>1$,
\begin{align*}J_\alpha: M^{\beta,p}\imply M^{\alpha+\beta,p}.\end{align*}
\end{propo}
\begin{proof}Once again, using $\int k_\alpha=1$, and proceeding as in the Lipschitz case,
\begin{align*}|J_\alpha f(x)-J_\alpha f(y)| &\leq \int_X |f(x)-f(z)||k_\alpha(x,z)-k_\alpha(y,z)|dm(z)\\
&\hspace{-2.75cm}\leq C\int_{B(x,2d(x,y))}d(x,z)^\beta(g(x)+g(z))\left(\frac{1}{d(x,z)^{N-\alpha}}+\frac{1}{d(y,z)^{N-\alpha}}\right)dm(z)\\
&\hspace{-2.25cm}+C\int_{B(x,2d(x,y))^c}d(x,z)^\beta(g(x)+g(z))\frac{d(x,y)}{d(x,z)^{N-\alpha+1}}dm(z)\\
&\hspace{-2.75cm}\leq Cg(x)d(x,y)^{\alpha+\beta}+Cd(x,y)^{\alpha+\beta}Mg(x)\\
&\hspace{-2.25cm} +Cd(x,y)^\beta g(x)d(x,y)^\alpha+Cd(x,y)^\beta Mg(y)d(x,y)^\alpha\\
&\hspace{-2.25cm} +Cd(x,y)g(x)\frac{1}{d(x,y)^{1-(\alpha+\beta)}}+Cd(x,y)\frac{1}{d(x,y)^{1-(\alpha+\beta)}}Mg(x)\\
&\hspace{-2.75cm}\leq Cd(x,y)^{\alpha+\beta}(Mg(x)+Mg(y)).\end{align*}
\end{proof}

\section{Potential spaces $L^{\alpha,p}$}

In this section we define potential spaces $L^{\alpha,p}$ and see they are Banach spaces. We prove they are embedded in certain Sobolev and Besov spaces, and that Lipschitz functions are dense. We finish the section with Sobolev embedding theorems for $L^{\alpha,p}$.

For $\alpha>0$, we define the \textbf{potential space}
\begin{align*}L^{\alpha,p}(X)=\{f\in L^p:\exists g\in L^p, f=J_\alpha g\}=J_\alpha(L^p)\end{align*}
and equip it with the following norm
\begin{align*}\|f\|_{\alpha,p}=\|f\|_p+\inf_{g\in J_\alpha^{-1}(\{f\})}\|g\|_p.\end{align*}

\begin{propo}$L^{\alpha,p}$ is Banach.\end{propo}
\begin{proof} To prove completeness, we will show the convergence of every absolutely convergent series. Let $(f_n)$ be a sequence in $L^{\alpha,p}$ such that 
\begin{align*}\sum_n\|f_n\|_{\alpha,p}<\infty.\end{align*}
In particular, $\sum_n\|f_n\|_p<\infty$, so the series $\sum_n f_n$ converges in $L^p$ to some function $f$. For each $n$, take $g_n$ in $L^p$ with  $f_n=J_\alpha g_n$ and
\begin{align*}\|g_n\|_p\leq \|f_n\|_{\alpha,p}+2^{-n},\end{align*}
then clearly $\sum_n \|g_n\|_p<\infty$ and $\sum_n g_n$ converges to some $g\in L^p$. Finally, as $J_\alpha$ is continuous in $L^p$,
\begin{align*}f=\sum_n f_n=\sum_n J_\alpha g_n=J_\alpha\left(\sum_n g_n\right)=J_\alpha g\end{align*}
so $f\in L^{\alpha,p}$, and
\begin{align*}\left\|f-\sum_{k=1}^nf_k\right\|_{\alpha,p}\leq \left\|f-\sum_{k=1}^nf_k\right\|_p+\left\|g-\sum_{k=1}^ng_k\right\|_p\imply 0.\end{align*}
\end{proof}

\begin{obs}\label{linf}$\|J_\alpha g\|_{\alpha,p}\leq 2\|g\|_p$, so it is continuous from $L^p$ onto $L^{\alpha,p}$. In particular, as $L^\infty\cap L^p$ is dense in $L^p$ for $1\leq p\leq\infty$, we get that $J_\alpha(L^\infty\cap L^p)$ is dense in $L^{\alpha,p}$.\end{obs}

The following theorem shows that `potential functions' have Haj\symbol{170}asz gradients, and this leads to some interesting results, such as Lipschitz density and embeddings in Sobolev spaces.

\begin{thm}\label{eleeme}Let $f=J_\alpha g$ for some $g$ such that $f$ is finite $a.e.$. Then if $0<\alpha<1$,
\begin{align*}|f(x)-f(y)|\leq C_\alpha d(x,y)^\alpha (Mg(x)+Mg(y))\end{align*}
for every $x,y$ outside a set of measure zero. If $\alpha\geq1$, then for each $\beta<1$ we get
\begin{align*}|f(x)-f(y)|\leq C_{\alpha,\beta}d(x,y)^\alpha (Mg(x)+Mg(y))\end{align*}
for every $x,y$ outside a set of measure zero.
\end{thm}
\begin{proof}Assume first $\alpha<1$. Let $d=d(x,y)$,
\begin{align*}|f(x)-f(y)| &\leq \int_X |g(z)||k_\alpha(x,z)-k_\alpha(y,z)|dm(z)\\
&\leq \int_{B(x,2d)}+\int_{X\menos B(x,2d)}=I+II.\end{align*}
In $I$ we have
\begin{align*}I &\leq C\int_{B(x,2d)}|g(z)|\frac{1}{d(x,z)^{N-\alpha}}dm(z)+C\int_{B(y,3d)}|g(z)|\frac{1}{d(y,z)^{N-\alpha}}dm(z)\\
&\leq Cd^\alpha (Mg(x)+Mg(y)),\end{align*}
and for $II$, as $d(x,z)\sim d(y,z)$ we get
\begin{align*}II &\leq Cd\int_{B(x,2d)^c}|g(z)|d(x,z)^{-(N+1-\alpha)}dm(z)\\
&\leq Cdd^{-(1-\alpha)}Mg(x)=Cd^\alpha Mg(x).\end{align*}
Let now $\alpha\geq 1$ and fix $0<\beta<1$. Observe that the bound for $I$ also holds in this case, and for $d(x,y)<1$ we get
\begin{align*}I\leq Cd^\alpha (Mg(x)+Mg(y))\leq Cd^\beta (Mg(x)+Mg(y)).\end{align*}
We now divide $X\menos B(x,2d)$ in two regions (and use in both cases the fact that $d(x,z)\sim d(y,z)$)
\begin{align*}\int_{2d\leq d(x,z)<5}|g(z)|\frac{d}{d(x,z)^{N-\alpha+1}}dm(z) &\leq\\
&\hspace{-1cm}\leq \int_{2d\leq d(x,z)<5}|g(z)|\frac{d^\beta}{d(x,z)^{N-(\alpha-\beta)}}dm(z)\\
&\hspace{-1cm}\leq Cd^\beta Mg(x);\end{align*}
and if $d(x,z)\geq 5$, as $d(y,z)\geq 4$ we can use the other bound for differences of $k_\alpha$ (property 6 in \ref{propk})
\begin{align*}\int_{d(x,z)\geq 5}|g(z)|\frac{d}{d(x,z)^{N+\alpha+1}}dm(z)\leq CdMg(x)\leq Cd^\beta Mg(x).\end{align*}
Finally, if $d(x,y)\geq 1$, as $|f|\leq Mg$,
\begin{align*}|f(x)-f(y)|\leq C(Mg(x)+Mg(y))\leq Cd(x,y)^\beta (Mg(x)+Mg(y)).\end{align*}
\end{proof}

\begin{coro}Let $1<p<\infty$. If $0<\alpha<1$, then $L^{\alpha,p}\embeds M^{\alpha,p}$. For $\alpha\geq 1$, $L^{\alpha,p}\embeds M^{\beta,p}$ for all $0<\beta<1$.\end{coro}

\begin{coro}\label{eleeme2}Let $p=\infty$. If $0<\alpha<1$, then $L^{\alpha,\infty}\embeds C^\alpha$. For $\alpha\geq 1$, $L^{\alpha,\infty}\embeds C^\beta$ for all $0<\beta<1$. In particular, functions in $L^{\alpha,\infty}$ are continuous for all $\alpha>0$ (after eventual modification on a null set).\end{coro}

From this last result and remark \ref{linf}, we get the following density property.

\begin{coro}Let $1\leq p\leq\infty$ and $\alpha>0$. Then $C^\beta\cap L^{\alpha,p}$ is dense in $L^{\alpha,p}$ for all $0<\beta\leq \alpha$ if $\alpha<1$, and for all $0<\beta<1$ if $\alpha\geq1$.\end{coro}

As a last corollary of theorem \ref{eleeme}, since $Mg$ is a Haj\symbol{170}asz gradient for potential functions, we get the following Poincaré inequality.

\begin{coro}Let $0<\alpha<1$ and $f=J_\alpha g$ for some $g$ such that $f\in L^1_{loc}$, then for each ball $B$ we get
\begin{align*}\fint_B |f-f_B|\leq C\text{diam}(B)^\alpha \fint_B Mg.\end{align*}
\end{coro}

Now, regarding Besov spaces, as $M^{\alpha,p}\embeds B^{\alpha-\epsilon}_{p,p}$ for $1\leq p<\infty$ and $0<\epsilon<\alpha$, from \ref{eleeme2} we obtain for $\alpha<1$ $L^{\alpha,p}\embeds B^{\alpha-\epsilon}_{p,p}$. This also holds true for $B^{\alpha-\epsilon}_{p,q}$. First, a lemma.

\begin{lem}\label{besovk}Let $0<\alpha<1$ and $q>0$ satifying $q(N-\alpha)<N<q(N+q-\alpha)$. Then there exists $C>0$ such that, for every $z\in X$ and $t>0$
\begin{align*}\int_X\fint_{B(x,t)}|k_\alpha(x,z)-k_\alpha(y,z)|^q dm(y)dm(x)\leq C t^{N-q(N-\alpha)}.\end{align*}
\end{lem}
\begin{proof} Consider
\begin{align*}A_1=\{(x,y):d(x,y)<t,d(x,z)<2t\};\end{align*}
\begin{align*}A_2=\{(x,y):d(x,y)<t,2t\leq d(x,z)\}.\end{align*}
Integrating over $A_1$, we get
\begin{align*}\iint_{A_1}\frac{1}{t^N}|k_\alpha(x,z)-k_\alpha(y,z)|^q dm(y)dm(x) &\leq C\int_{B(z,3t)}|k_\alpha(x,z)|^q dm(x)\\
&\leq C\int_{B(z,3t)}\frac{1}{d(x,z)^{q(N-\alpha)}}dm(x)\\
&\leq Ct^{N-q(N-\alpha)},\end{align*}
and the last inequality holds because $N>q(N-\alpha)$.

In $A_2$ we have $d(x,z)\sim d(y,z)$, and then, as $d(x,y)<t$,
\begin{align*}\iint_{A_2}\frac{1}{t^N}|k_\alpha(x,z)-k_\alpha(y,z)|^qdm(y)dm(x) &\leq\\
&\hspace{-2cm}\leq Ct^q\iint_{A_2}\frac{1}{t^N}\frac{1}{d(x,z)^{q(N+1-\alpha)}}dm(y)dm(x)\\
&\hspace{-2cm}\leq Ct^q\int_{X\menos B(z,2t)}\frac{1}{d(x,z)^{q(N+1-\alpha)}}dm(x)\\
&\hspace{-2cm}\leq C t^qt^{N-q(N+1-\alpha)}\leq Ct^{N-q(N-\alpha)},\end{align*}
given $N<q(N+1-\alpha)$.
\end{proof}

\begin{propo}Let $f=J_\alpha g$, $0<\alpha<1$ and $1\leq p\leq\infty$, then for $t>0$ we get
\begin{align*}E_pf(t)\leq Ct^\alpha\|g\|_p\end{align*}\end{propo}
\begin{proof} If $p<\infty$,
\begin{align*}|f(x)-f(y)|^p &\leq\left(\int_X|k_\alpha(x,z)-k_\alpha(y,z)|^{\frac{1}{p}+\frac{1}{p'}}|g(z)|dm(z)\right)^p\\
&\leq\left(\int_X|k_\alpha(x,z)-k_\alpha(y,z)||g(z)|^pdm(z)\right)\\
&\hspace{1cm}\times\left(\int_X|k_\alpha(x,z)-k_\alpha(y,z)|dm(z)\right)^{p/p'}.\end{align*}
By lemma \ref{ineqk} for $q=1$, as $d(x,y)<t$ and $\alpha<1$,
\begin{align*}\int_X|k_\alpha(x,z)-k_\alpha(y,z)|dm(z)\leq Ct^\alpha\end{align*}
so
\begin{align*}\int_X\fint_{B(x,t)}|f(x)-f(y)|^pdm(y)dm(x) &\leq\\
&\hspace{-6cm}\leq C t^{\alpha p/p'}\int_X\left(\int_X\fint_{B(x,t)}|k_\alpha(x,z)-k_\alpha(y,z)|dm(y)dm(x)\right)|g(z)|^pdm(z)\end{align*}
and by lemma \ref{besovk} (also taking $q=1$)
\begin{align*}\int_X\fint_{B(x,t)}|f(x)-f(y)|^pdm(y)dm(x)\leq Ct^{\alpha p/p'} t^\alpha \|g\|_p^p=C t^{\alpha p}\|g\|_p^p.\end{align*}
For $p=\infty$, as $\alpha<1$,
\begin{align*}E_\infty f(t) &=\sup_{d(x,y)<t}|f(x)-f(y)|\\
&\leq C\sup_{d(x,y)<t} d(x,y)^\alpha(Mg(x)+Mg(y))\\
&\leq Ct^\alpha\|g\|_\infty.\end{align*}
\end{proof}

We can now conclude the following embedding in Besov spaces.

\begin{coro}Let $1\leq p\leq\infty$ and $0<\alpha<1$. Then for $1\leq q<\infty$ and $0<\epsilon<\alpha$ we have $L^{\alpha,p}\embeds B^{\alpha-\epsilon}_{p,q}$. For $q=\infty$ we obtain $L^{\alpha,p}\embeds B^\alpha_{p,\infty}$.\end{coro}
\begin{proof}Let $f=J_\alpha g$. By the previous proposition, if $q=\infty$,
\begin{align*}\|f\|_{B^\alpha_{p,\infty}}=\|f\|_p+\sup_{t>0}t^{-\alpha}E_pf(t)\leq C\|f\|_{\alpha,p}.\end{align*}
And for $1\leq q<\infty$, as we also have $E_p f\leq C\|f\|_p$,
\begin{align*}\|f\|_{B^{\alpha-\epsilon}_{p,q}} &\leq C\|f\|_p+C\left(\int_0^1 t^{-(\alpha-\epsilon)q}E_pf(t)^q\frac{dt}{t}\right)^{1/q}\\
&\leq C\|f\|_p+C\|g\|_p\left(\int_0^1 t^{\epsilon q}\frac{dt}{t}\right)\leq \frac{C}{\epsilon^{1/q}}\|f\|_{\alpha,p}.\end{align*}
\end{proof}

We finish this section with Sobolev-type embedding theorems for potential spaces. First we need a lemma.

\begin{lem}\label{boundk}For $\alpha>0$ and $q>0$ satisfying $q(N-\alpha)<N<q(N+\alpha)$, there exists $C>0$ such that for every $x\in X$,
\begin{align*}\int_X k_\alpha(x,y)^qdm(y)\leq C<\infty.\end{align*}
\end{lem}
\begin{proof} By lemma \ref{propk},
\begin{align*}k_\alpha(x,y)^q\leq C\frac{\chi_{B(x,4)}(y)}{d(x,y)^{q(N-\alpha)}}+C\frac{\chi_{X\menos B(x,4)}(y)}{d(x,y)^{q(N+\alpha)}},\end{align*}
and restrictions over $q$ guarantee integrability.
\end{proof}

\begin{thm}Let $1<p<\infty$ and $\alpha>0$. The following embeddings hold for $L^{\alpha,p}$
\begin{description} 
	\item[a.] If $p<\frac{N}{\alpha}$,
	\begin{align*}L^{\alpha,p}\embeds L^q\end{align*}
	for $p\leq q\leq p^*$ where $\frac{1}{p^*}=\frac{1}{p}-\frac{\alpha}{N}$.
	\item[b.] If $p=\frac{N}{\alpha}$, then for $p\leq q<\infty$,
	\begin{align*}L^{\alpha,p}\embeds L^q.\end{align*}
	If in addition $\alpha<1$,
	\begin{align*}L^{\alpha,p}\embeds BMO.\end{align*}
	\item[c.] If $p>\frac{N}{\alpha}$,then for $p\leq q\leq \infty$
	\begin{align*}L^{\alpha,p}\embeds L^q.\end{align*}
	If in addition $\alpha<1+N/p$,
	\begin{align*}L^{\alpha,p}\embeds C^{\alpha-N/p}.\end{align*}
\end{description}
\end{thm}
\begin{proof}
\begin{description} 
	\item[a.] We know $L^{\alpha,p}\embeds L^p$ (for $\|f\|_p\leq \|f\|_{\alpha,p}$), then if we prove $L^{\alpha,p}\embeds L^{p*}$, by an interpolation argument we are done. This follows from $|J_\alpha f|\leq CI_\alpha|f|$, as $(N-\alpha)p'>N$ and for any $t>0$ we get
\begin{align*}|J_\alpha f(x)| &\leq C\int_{B(x,t)}\frac{|f(y)|}{d(x,y)^{N-\alpha}}dm(y)+C\int_{X\menos B(x,t)}\frac{|f(y)|}{d(x,y)^{N-\alpha}}dm(y)\\
&\leq Ct^\alpha Mf(x)+Ct^{(N-(N-\alpha)p')/p'}\|f\|_p\\
&= Ct^\alpha Mf(x)+Ct^{-N/p^*}\|f\|_p.\end{align*}
This last expression attains its minimum for $t=CMf(x)^{-p/N}\|f\|_p^{p/N}$, and for this value of $t$ we obtain
\begin{align*}|J_\alpha f(x)|\leq C Mf(x)^{p/p^*}\|f\|_p^{1-p/p^*},\end{align*}
and as $p>1$, boundedness of the maximal function implies
\begin{align*}\int_X |J_\alpha f|^{p^*}dm\leq C\|f\|_p^{p^*-p}\int_X (Mf)^pdm\leq C \|f\|_p^{p^*}.\end{align*}

	\item[b.] Let $N/\alpha=p<q<\infty$, so there exists $a>1$ such that
	\begin{align*}1+\frac{1}{q}=\frac{\alpha}{N}+\frac{1}{a}.\end{align*}
	In particular $a(N-\alpha)<N$ (and also $N<a(N+\alpha)$, as $a>1$), so by the previous lemma
	\begin{align*}\int_X k_\alpha(x,y)^a dm(y)\leq C<\infty.\end{align*}
	Let now $f=J_\alpha g$ with $g\in L^p$, as $\frac{1}{q'}=\frac{1}{p'}+\frac{1}{a'}$ by Hölder's inequality we obtain the following Young-type inequality
	\begin{align*}|f(x)| &\leq \int_X k_\alpha(x,y)^{a/q+a/p'}|g(y)|^{p/q+p/a'}dm(y)\\
				&\leq\left(\int_X k_\alpha(x,y)^a|g(y)|^pdm(y)\right)^{1/q}\left(\int_X|g(y)|^pdm(y)\right)^{1/a'}\\
				&\hspace{1cm}\times\left(\int_Xk_\alpha(x,y)^adm(y)\right)^{1/p'}\\
				&\leq C\|g\|_p^{p/a'}\left(\int_X k_\alpha(x,y)^a|g(y)|^pdm(y)\right)^{1/q}\end{align*}
	(here we use $a/q+a/p'=1$ and $p/q+p/a'=1$) and
	\begin{align*}\int_X|f(x)|^qdm(x) &\leq C\|g\|_p^{qp/a'}\int_X\int_X k_\alpha(x,y)^a|g(y)|^pdm(y)dm(x)\\
		&\leq C\|g\|_p^{p(q/a'+1)}=C\|g\|_p^q.\end{align*}
	
	Moreover, if $\alpha<1$, by Poincaré inequality for any ball $B$,
	\begin{align*}\fint_B |f-f_B| &\leq C\text{diam}(B)^\alpha \fint_B Mg\leq Cm(B)^{\alpha/N}\left(\fint_B (Mg)^{N/\alpha}\right)^{\alpha/N}\\
		&\leq C\left(\int_B (Mg)^{N/\alpha}\right)^{\alpha/N}\leq C\|g\|_{N/\alpha}\end{align*}
	and we conclude 
	\begin{align*}\|f\|_{BMO}\leq C\|f\|_{\alpha,N/\alpha}.\end{align*}
	
	\item[c.] For the first part, again by interpolation it is enough to prove $L^{\alpha,p}\embeds L^\infty$. If $f=J_\alpha g$ with $g\in L^p$,
	\begin{align*}|f(x)| &=|J_\alpha g(x)|\leq \int_X k_\alpha(x,y)|g(y)|dm(y)\\
		&\leq \|g\|_p\left(\int_X k_\alpha(x,y)^{p'}dm(y)\right)^{1/p'}\\
		&\leq C\|g\|_p\leq C\|f\|_{\alpha,p}\end{align*}
	as long as $p'(N-\alpha)<N<p'(N+\alpha)$. The second inequality is trivial for $p'\geq 1$ and the first one is equivalent to $p\alpha>N$.
	
	Assume now $\alpha<1+N/p$. Then
	\begin{align*}|f(x)-f(y)| &\leq \int_X |k_\alpha(x,z)-k_\alpha(y,z)||g(z)|dm(z)\\
		&\leq \|g\|_p\left(\int_X |k_\alpha(x,z)-k_\alpha(y,z)|^{p'}dm(z)\right)^{1/p'}\\
		&\leq C\|g\|_p d(x,y)^{\frac{N-p'(N-\alpha)}{p'}}=C\|g\|_p d(x,y)^{\alpha-N/p}\end{align*}
	if $p'(N-\alpha)<N<p'(N-\alpha+1)$. The first inequality is once again equivalent to $p>N/\alpha$, and the second to $\alpha<1+N/p$.
\end{description}
\end{proof}

\section{The inverse of $J_\alpha$}

In this section, with the fractional derivative $D_\alpha$ as defined in \cite{GSV}, we prove conditions for the composition $(I+D_\alpha)J_\alpha$ to be inversible in $L^p$ for $1<p<\infty$, which in turn will lead to inversibility of $J_\alpha$. We follow the techniques used in \cite{Hz}, proving
\begin{align*}\|I-(I+D_\alpha)J_\alpha\|_{L^p\imply L^p}<1\end{align*}
by rewriting the operators in terms of $(Q_t)_{t>0}$ instead of $(S_t)_{t>0}$, and applying the $T1$ theorem for Ahlfors spaces (see \cite{Ga}).

Let $\alpha>0$. Define
\begin{align*}n_\alpha(x,y)=\int_0^\infty \alpha t^{-\alpha}s(x,y,t)\frac{dt}{t}.\end{align*}
This kernel satisfies
\begin{align*}n_\alpha(x,y)\sim \frac{1}{d(x,y)^{N+\alpha}}\end{align*}
and
\begin{align*}|n_\alpha(x,y)-n_\alpha(x',y)|\leq C d(x,x')(d(x,y)\yy d(x',y))^{-(N+1+\alpha)}.\end{align*}

The fractional derivative can be then defined as
\begin{align*}D_\alpha f(x)=\int_X n_\alpha(x,y)(f(x)-f(y))dm(y),\end{align*}
see for instance \cite{GSV}, whenever this integral makes sense (for instance if $f$ has sufficient regularity of Lipschitz or Besov type).

Let us now rewrite the operators with  $Q_t=-t\frac{d}{dt}S_t$. Assume $f\in C_c^\gamma$ for some $\alpha<\gamma\leq 1$, then
\begin{align*}J_\alpha f(x) &=\int_X k_\alpha(x,y)f(y)dm(y)=\int_X\int_0^\infty \frac{\alpha t^{\alpha-1}}{(1+t^\alpha)^2}s(x,y,t)f(y)dtdm(y)\\
&=\int_0^\infty \frac{\alpha t^{\alpha-1}}{(1+t^\alpha)^2}S_tf(x)dt=\int_0^\infty \frac{d}{dt}\left(\frac{1}{1+t^{-\alpha}}\right)S_tf(x)dt\\
&=\left.\frac{S_tf(x)}{1+t^{-\alpha}}\right|_0^\infty+\int_0^\infty \frac{1}{1+t^{-\alpha}}\left(-t\frac{d}{dt}S_tf(x)\right)\frac{dt}{t}\\
&=\int_0^\infty \frac{1}{1+t^{-\alpha}}Q_tf(x)\frac{dt}{t}\end{align*}
where we have used $S_tf\imply f$ when $t\imply 0$ and $S_tf\imply 0$ when $t\imply\infty$.

On the other hand, we obtain
\begin{align*}D_\alpha f(x) &=\int_X n_\alpha(x,y)(f(x)-f(y))dm(y)\\
&=\int_X\int_0^\infty \alpha t^{-\alpha-1}s(x,y,t)(f(x)-f(y))dtdm(y)\\
&=\int_0^\infty \alpha t^{-\alpha-1}(f(x)-S_tf(x))dt=\int_0^\infty \frac{d}{dt}\left(t^{-\alpha}\right)(S_tf(x)-f(x))dt\\
&=\left.\frac{(S_tf(x)-f(x))}{t^\alpha}\right|_0^\infty+\int_0^\infty t^{-\alpha}\left(-t\frac{d}{dt}S_tf(x)\right)\frac{dt}{t}\\
&=\int_0^\infty t^{-\alpha}Q_tf(x)\frac{dt}{t},\end{align*}
where we have used that $S_tf\imply 0$ when $t\imply\infty$ and that $|S_tf(x)-f(x)|\leq Ct^\gamma$. Since we also have
\begin{align*}f(x)=-\int_0^\infty\frac{d}{dt}S_tf(x)dt=\int_0^\infty Q_tf(x)\frac{dt}{t},\end{align*}
we get
\begin{align*}(I+ D_\alpha)f(x)=\int_0^\infty (1+t^{-\alpha})Q_tf(x)\frac{dt}{t}.\end{align*}

This way,
\begin{align*}(I+D_\alpha) J_\alpha f=\int_0^\infty\int_0^\infty\frac{1+s^{-\alpha}}{1+t^{-\alpha}}Q_sQ_tf\frac{dt}{t}\frac{ds}{s},\end{align*}
and as we also have
\begin{align*}f=\int_0^\infty\int_0^\infty Q_sQ_tf\frac{dt}{t}\frac{ds}{s},\end{align*}
we conclude
\begin{align*}(I-(I+D_\alpha)J_\alpha)f &=\int_0^\infty\int_0^\infty \left(1-\frac{1+s^{-\alpha}}{1+t^{-\alpha}}\right) Q_sQ_tf\frac{dt}{t}\frac{ds}{s}\\
&=\int_0^\infty\int_0^\infty \frac{t^{-\alpha}-s^{-\alpha}}{1+t^{-\alpha}} Q_sQ_tf\frac{dt}{t}\frac{ds}{s}\\
&=\int_0^\infty (1-v^\alpha)\left(\int_0^\infty \frac{1}{1+(uv)^\alpha}Q_uQ_{uv}f\frac{du}{u}\right)\frac{dv}{v}.\end{align*}

For each $v>0$ we define
\begin{align*}T_{\alpha,v}f=\int_0^\infty \frac{1}{1+(uv)^\alpha}Q_uQ_{uv}f\frac{du}{u},\end{align*}
and, following \cite{Hz}, if we can prove
\begin{align*}\|T_{\alpha,v}f\|_p\leq C_{\alpha,p}(v)\|f\|_p,\end{align*}
with
\begin{align*}\int_0^\infty |1-v^\alpha| C_{\alpha,p}(v)\frac{dv}{v}<1\end{align*}
for $\alpha$ small enough, we will obtain
\begin{align*}\|(I-(I+D_\alpha) J_\alpha)f\|_p\leq \int_0^\infty |1-v^\alpha|\|T_{v,\alpha}f\|_p\frac{dv}{v}< \|f\|_p\end{align*}
and therefore $(I+D_\alpha) J_\alpha$ will be inversible for those values of $\alpha$.

To prove the boundedness of $T_{\alpha,v}$, we will use the $T1$ theorem as presented in \ref{tedeuno}. As a first step, we need to show $T_{\alpha,v}$ is a singular integral operator, for which we need to find its kernel.

\begin{lem}For $u,v>0$, $x,z\in X$,
\begin{align*}\left|\int_X q(x,y,u)q(y,z,uv)dm(y)\right|\leq C\left(v\yy\frac{1}{v^{N+1}}\right)\frac{1}{u^N}\chi_{\left(\frac{d(x,z)}{4(v+1)},\infty\right)}(u).\end{align*}
As a consequence,
\begin{align*}\left|\int_0^\infty \frac{1}{1+(uv)^\alpha}\int_X q(x,y,u)q(y,z,uv)dm(y)\frac{du}{u}\right|\leq C\left(v\yy\frac{1}{v}\right) \frac{1}{d(x,z)^N}.\end{align*}
\end{lem}
\begin{proof}The second inequality follows immediately from the first one. For this one, as
\begin{align*}q(x,y,u)=0 \text{ when } d(x,y)\geq 4u; \quad q(y,z,uv)=0 \text{ when } d(y,z)\geq 4uv,\end{align*}
for the product to be non zero $d(x,z)<4u(v+1)$ must hold. If $v\geq 1$, as $\int_X q(x,y,u)q(x,z,uv)dm(y)=0$,
\begin{align*}\left|\int_X q(x,y,u)q(y,z,uv)dm(y)\right| &=\\
&\hspace{-1cm}=\left|\int_X q(x,y,u)(q(y,z,uv)-q(x,z,uv))dm(y)\right|\\
&\hspace{-1cm}\leq C\int_{B(x,4u)}\frac{1}{u^N}\frac{d(x,y)}{(uv)^{N+1}}dm(y)\leq C\frac{1}{u^N}\frac{1}{v^{N+1}};\end{align*}
and if $v<1$, as $\int_X q(x,z,u)q(y,z,uv)dm(y)=0$,
\begin{align*}\left|\int_X q(x,y,u)q(y,z,uv)dm(y)\right| &=\\
&\hspace{-1cm}=\left|\int_X (q(x,y,u)-q(x,z,u))q(y,z,uv)dm(y)\right|\\
&\hspace{-1cm}\leq C\int_{B(z,4uv)}\frac{d(y,z)}{u^{N+1}}\frac{1}{(uv)^N}dm(y)\leq C\frac{1}{u^N}v.\end{align*}
\end{proof}

Let now $f,g\in C^\beta_c$ with disjoint supports, and let $x\in supp(g)$. Then
\begin{align*}T_{\alpha,v}f(x) &=\int_0^\infty \frac{1}{1+(uv)^\alpha}Q_uQ_{uv}f\frac{du}{u}\\
&\hspace{-0.75cm}=\int_0^\infty \frac{1}{1+(uv)^\alpha}\left(\int_X q(x,y,u)\left(\int_X q(y,z,uv)f(z)dm(z)\right)dm(y)\right)\frac{du}{u}\end{align*}
and from the previous lemma we have this integral converges absolutely, so we can change the order of integration and obtain
\begin{align*}\left\langle T_{\alpha,v}f,g\right\rangle=\int_X\int_X N_{\alpha,v}(x,z)f(z)g(x)dm(z)dm(x),\end{align*}
where
\begin{align*}N_{\alpha,v}(x,z)=\int_0^\infty \frac{1}{1+(uv)^\alpha}\int_X q(x,y,u)q(y,z,uv)dm(y)\frac{du}{u}.\end{align*}

From the previous lemma, $N_{\alpha,v}(x,z)\leq C\left(v\yy \frac{1}{v}\right)\frac{1}{d(x,z)^N}$. To see that $T_{\alpha,v}$ is a singular integral operator we need to check the smoothness conditions for the kernel $N_{\alpha,v}$.

\begin{lem}For $u,v>0$, $x,x',z\in X$ and $0<\delta<1$, it holds
\begin{align*}\left|\int_X (q(x,y,u)-q(x',y,u))q(y,z,uv)dm(y)\right| &\leq\\
&\hspace{-5cm}\leq C\left(\frac{d(x,x')}{u}\right)^{1-\delta}\left(v^\delta\yy \frac{1}{v^{N+1}}\right)\frac{1}{u^N}\chi_{\left(\frac{d(x,z)\yy d(x',z)}{4(v+1)},\infty\right)}(u).\end{align*}
From this we obtain
\begin{align*}\left|\int_0^\infty \frac{1}{1+(uv)^\alpha}\int_X (q(x,y,u)-q(x',y,u))q(y,z,uv)dm(y)\frac{du}{u}\right| &\leq\\
&\hspace{-7cm}\leq C\frac{d(x,x')^{1-\delta}}{(d(x,z)\yy d(x',z))^{N+1-\delta}}\left(v\yy \frac{1}{v}\right)^\delta.\end{align*}
\end{lem}
\begin{proof}As in the other lemma, the second inequality follows from the first one. We consider two cases:
If $v\geq1$ y $d(x,x')\geq u$, by that same lemma,
\begin{align*}\left|\int_X (q(x,y,u)-q(x',y,u))q(y,z,uv)dm(y)\right| &\leq\\
&\hspace{-3cm}\leq C\frac{1}{v^{N+1}}\frac{1}{u^N}\left(\chi_{\left(\frac{d(x,z)}{4(v+1)},\infty\right)}(u)+\chi_{\left(\frac{d(x',z)}{4(v+1)},\infty\right)}(u)\right)\\
&\hspace{-3cm}\leq C\frac{1}{v^{N+1}}\frac{1}{u^N}\chi_{\left(\frac{d(x,z)\yy d(x',z)}{4(v+1)},\infty\right)}(u)\\
&\hspace{-3cm}\leq C\frac{1}{v^{N+1}}\frac{1}{u^N}\chi_{\left(\frac{d(x,z)\yy d(x',z)}{4(v+1)},\infty\right)}(u) \left(\frac{d(x,x')}{u}\right)^{1-\delta}.\end{align*}
And for $d(x,x')<u$, the integrand will be nonzero only if $d(x,z)<4u(v+1)$ or $d(x',z)<4u(v+1)$, so
\begin{align*}\left|\int_X (q(x,y,u)-q(x',y,u))q(y,z,uv)dm(y)\right|=\\
&\hspace{-6cm}=\left|\int_X (q(x,y,u)-q(x',y,u))(q(y,z,uv)-q(x,z,uv))dm(y)\right|\\
&\hspace{-6cm}\leq Cd(x,x')\frac{1}{u^{N+1}}\frac{1}{(uv)^{N+1}}\int_{B(x,4u)\cup B(x',4u)}d(x,y)dm(y)\\
&\hspace{-6cm}\leq C\frac{1}{v^{N+1}}\frac{1}{u^N}\chi_{\left(\frac{d(x,z)\yy d(x',z)}{4(v+1)},\infty\right)}(u) \left(\frac{d(x,x')}{u}\right)\\
&\hspace{-6cm}\leq C\frac{1}{v^{N+1}}\frac{1}{u^N}\chi_{\left(\frac{d(x,z)\yy d(x',z)}{4(v+1)},\infty\right)}(u) \left(\frac{d(x,x')}{u}\right)^{1-\delta}.\end{align*}
For the case $v<1$, on one hand by the previous lemma we obtain
\begin{align*}\left|\int_X (q(x,y,u)-q(x',y,u))q(y,z,uv)dm(y)\right|\leq\\
&\hspace{-1cm}\leq Cv\frac{1}{u^N}\chi_{\left(\frac{d(x,z)\yy d(x',z)}{4(v+1)},\infty\right)}(u),\end{align*}
on the other hand
\begin{align*}\left|\int_X (q(x,y,u)-q(x',y,u))q(y,z,uv)dm(y)\right|\leq\\
&\hspace{-2cm}\leq C \frac{d(x,x')}{u}\frac{1}{u^N}\chi_{\left(\frac{d(x,z)\yy d(x',z)}{4(v+1)},\infty\right)}(u),\end{align*}
and by combining both inequalities we get
\begin{align*}\left|\int_X (q(x,y,u)-q(x',y,u))q(y,z,uv)dm(y)\right| &\leq\\
&\hspace{-2cm}\leq C v^\delta \left(\frac{d(x,x')}{u}\right)^{1-\delta}\frac{1}{u^N}\chi_{\left(\frac{d(x,z)\yy d(x',z)}{4(v+1)},\infty\right)}(u).\end{align*}
\end{proof}

For the rest of the section, we fix $0<\delta<1$. Joining both lemmas we conclude

\begin{thm}\label{singint}$T_{\alpha,v}$ is a singular integral operator. Its kernel $N_{\alpha,v}$ satisfies
\begin{align*}|N_{\alpha,v}(x,z)|\leq C\left(v\yy\frac{1}{v}\right)^\delta \frac{1}{d(x,z)^N};\end{align*}
and for $3d(x,x')<d(x,z)$,
\begin{align*}|N_{\alpha,v}(x,z)-N_{\alpha,v}(x',z)|\leq C\left(v\yy\frac{1}{v}\right)^\delta\frac{d(x,x')^{1-\delta}}{d(x,z)^{N+1-\delta}}\end{align*}
and
\begin{align*}|N_{\alpha,v}(z,x)-N_{\alpha,v}(z,x')|\leq C\left(v\yy\frac{1}{v}\right)^\delta\frac{d(x,x')^{1-\delta}}{d(x,z)^{N+1-\delta}}.\end{align*}
\end{thm}

To prove each $T_{\alpha,v}$ is a Calderón-Zygmund operator, and thus bounded in $L^p$, we will use the $T1$ theorem. The next lemma proves the other conditions needed.

\begin{lem}\label{weak}$T_{\alpha,v}$ satisfies
\begin{align*}T_{\alpha,v}1=0,\end{align*}
\begin{align*}T_{\alpha,v}^*1=0,\end{align*}
and for $f,g\in C^\beta_c(B)$, for some ball $B$,
\begin{align*}|\langle T_{\alpha,v}f,g\rangle|\leq C \left(v\yy\frac{1}{v}\right)^\delta m(B)^{1+\frac{2\beta}{N}}[f]_\beta[g]_\beta.\end{align*}
\end{lem}
\begin{proof}The first equality is immediate, the second uses the fact that $q$ is symmetrical.
\begin{align*}\langle T_{\alpha,v}f,g\rangle &= \int_X\left(\int_X N_{\alpha,v}(x,z)f(z)dm(z)\right)g(x)dm(x)\\
&=\int_X \int_X\int_0^\infty \int_X \frac{1}{1+(uv)^\alpha}\\
&\hspace{1cm}\times q(x,y,u)q(y,z,uv)dm(y)\frac{du}{u}f(z)dm(z)g(x)dm(x)\\
&=\int_Xf(z)\left(\int_X N^*_{\alpha,v}(z,x)g(x)dm(x)\right)dm(z)=\langle f,T_{\alpha,v}^*g\rangle\end{align*}
so clearly $T_{\alpha,v}^*1=0$.

For the third one, as
\begin{align*}\left\langle T_{\alpha,v}f,g \right\rangle &=\\
&\hspace{-1.5cm}=\int\limits_0^\infty\frac{1}{1+(uv)^\alpha}\int_X\int_X\int_X q(x,y,u)q(y,z,uv)f(z)g(x)dm(y)dm(z)dm(x)\frac{du}{u}\end{align*}
we observe that the triple integral inside may be estimated in three different ways
\begin{itemize}
	\item Firstly,
	\begin{align*}A &=\left|\int_X\int_X\int_X q(x,y,u)q(y,z,uv)f(z)g(x)dm(y)dm(z)dm(x)\right|\\
	&\leq C\|f\|_\infty\|g\|_\infty\left(v\yy\frac{1}{v^{N+1}}\right)\frac{1}{u^N}\int_B\int_B \chi_{B(x,4u(v+1))}(z)dm(z)dm(x)\\
	&\leq C[f]_\beta[g]_\beta m(B)^{2\beta/N}\left(v\yy\frac{1}{v^{N+1}}\right)m(B)(v+1)^N\\
	&\leq C\left(v\yy\frac{1}{v}\right)[f]_\beta[g]_\beta m(B)^{1+2\beta/N}.\end{align*}
\item Secondly, using the fact that $\int_X q(x,y,u)q(y,z,uv)f(y)g(x)dm(z)=0$,
	\begin{align*} A&=\left|\int_X\int_X\int_X q(x,y,u)q(y,z,uv)(f(z)-f(y))g(x)dm(z)dm(y)dm(x)\right|\\
	&\leq C[f]_\beta\|g\|_\infty \int_B\fint_{B(x,4u)}\fint_{B(y,4uv)}d(z,y)^\beta dm(z)dm(y)dm(x)\\
	&\leq C[f]_\beta[g]_\beta m(B)^{1+\beta/N}(uv)^\beta\\
	&\leq C\left(\frac{uv}{m(B)^{1/N}}\right)^\beta [f]_\beta[g]_\beta m(B)^{1+2\beta/N}.\end{align*}
\item And lastly, it also holds
	\begin{align*}A &\leq C\|f\|_\infty\|g\|_\infty \frac{m(B)^2}{(uv)^N}\\
	&\leq C\left(\frac{uv}{m(B)^{1/N}}\right)^{-N} [f]_\beta[g]_\beta m(B)^{1+2\beta/N}.\end{align*}
\end{itemize}

By taking an appropriate combination of the previous three inequalities, we have
	\begin{align*}A &=\left|\int_X\int_X\int_X q(x,y,u)q(y,z,uv)f(z)g(x)dm(y)dm(z)dm(x)\right|\\
	&\leq C\left(v\yy\frac{1}{v}\right)^\delta \left(\left(\frac{uv}{m(B)^{1/N}}\right)^\beta\yy \left(\frac{uv}{m(B)^{1/N}}\right)^{-N}\right)^{1-\delta}[f]_\beta[g]_\beta m(B)^{1+2\beta/N},\end{align*}
and conclude
	\begin{align*}\left|\left\langle T_{\alpha,v}f,g \right\rangle\right|\leq C\left(v\yy\frac{1}{v}\right)^\delta [f]_\beta[g]_\beta m(B)^{1+2\beta/N}.\end{align*}
\end{proof}

Thus the $T1$ theorem holds for each $T_{\alpha,v}$, and we get

\begin{thm} For $1<p<\infty$ and $0<\delta<1$ the following holds
\begin{align*}\|T_{\alpha,v}f\|_p\leq C_p\left(v\yy \frac{1}{v}\right)^\delta\|f\|_p.\end{align*}
\end{thm}

The fact that the $L^p$-constant of $T_{\alpha,v}$ is bounded by the constants appearing in Theorem \ref{singint} and Lemma \ref{weak} follows the same ideas that the Euclidean case (see for instance \cite{Gr}).

From this result, as for $\alpha<\delta$ we have
\begin{align*}\|I-(I+D_\alpha)J_\alpha\|_{L^p\imply L^p}\leq \int_0^\infty |1-v^\alpha|\|T_{\alpha,v}\|_{L^p\imply L^p}\frac{dv}{v}\leq C_p\frac{\alpha}{\delta^2-\alpha^2}\end{align*}
so we obtain the estimate we were looking for and we can conclude
\begin{itemize}
	\item For any $0<\alpha<1$,  $I-(I+D_\alpha)J_\alpha$, and thus $(I+D_\alpha)J_\alpha$, is bounded in $L^p$
	\item There exists $\alpha_0<1$ such that, for $\alpha<\alpha_0$,
\begin{align*}\|I-(I+D_\alpha)J_\alpha\|_{L^p\imply L^p}<1,\end{align*}
and thus $(I+D_\alpha)J_\alpha$ is inversible (with bounded inverse) in $L^p$. As $J_\alpha$ maps $L^p$ \textsl{onto} $L^{\alpha,p}$, 
\begin{align*}\left[(I+D_\alpha)J_\alpha\right]^{-1}(I+D_\alpha)J_\alpha = Id_{L^p}\end{align*}
so $J_\alpha$ is inversible with inverse $J_\alpha^{-1}:L^{\alpha,p}\imply L^p$ given by
\begin{align*}J_\alpha^{-1}=\left[(I+D_\alpha)J_\alpha\right]^{-1}(I+D_\alpha).\end{align*}
\end{itemize}

\section{A characterization of $L^{\alpha,p}$ in terms of $D_\alpha$}

For $0<\alpha<1$ and $1<p<\infty$, we proved that, if $f\in L^{\alpha,p}$, then $f\in L^p$ (this holds for any $\alpha>0$ and $1\leq p\leq\infty$) and $(I+D_\alpha)f\in L^p$, so
\begin{align*}\text{If } f\in L^{\alpha,p}\text{, then } f,D_\alpha f\in L^p,\end{align*}
moreover,
\begin{align*}\|D_\alpha f\|_p\leq C\|f\|_{\alpha,p}.\end{align*}

For the case $\alpha<\alpha_0$, we obtain the reciprocal.

\begin{thm}Let $1<p<\infty$ and $0<\alpha<\alpha_0$. Then
\begin{align*}f\in L^{\alpha,p}\text{ if and only if } f,D_\alpha f\in L^p,\end{align*}
Furthermore,
\begin{align*}\|f\|_{\alpha,p}\sim \|(I+D_\alpha)f\|_p.\end{align*}
\end{thm}
\begin{proof} We have already seen in this case $J_\alpha:L^p\imply L^{\alpha,p}$ is bijective, and therefore $I+D_\alpha$ is also bijective. If $f,D_\alpha f\in L^p$, define 
\begin{align*}g=\left[(I+D_\alpha)J_\alpha\right]^{-1}(I+D_\alpha)f,\end{align*}
we get $g\in L^p$ and
\begin{align*}J_\alpha g= J_\alpha\left[(I+D_\alpha)J_\alpha\right]^{-1}(I+D_\alpha)f=J_\alpha J_\alpha^{-1}(I+D_\alpha)^{-1}(I+D_\alpha)f=f.\end{align*}
We also get
\begin{align*}\|f\|_{\alpha,p} &=\|f\|_p+\|J_\alpha^{-1}f\|_p\\
&\leq C\|J_\alpha^{-1}f\|_p=C\|\left[(I+D_\alpha)J_\alpha\right]^{-1}(I+D_\alpha) f\|_p\\
&\leq C\|(I+D_\alpha)f\|_p.\end{align*}\end{proof}

We can also characterize functions in $L^{\alpha,p}$ in terms of the Riesz potential $I_\alpha$ as follows. In \cite{GSV} and \cite{Ga}, it is proven there exists $0<\tilde{\alpha}_0$ such that, for $\alpha<\tilde{\alpha}_0$, the operator $D_\alpha I_\alpha$ is inversible in $L^p$, $1<p<\infty$. Thus we obtain

\begin{coro}For $\alpha>0$ satisfying $\alpha<\alpha_0\yy \tilde{\alpha}_0$ and $1<p<\infty$, we get
\begin{align*}f\in L^{\alpha,p}\text{ if and only if } f\in L^p \text{ and there exists $\gamma\in L^p$ with }f=I_\alpha \gamma.\end{align*}\end{coro}

As another corolary, the following embeddings hold, which follow from the fact that $D_\alpha f\in L^p$ for $f$ smooth enough.

\begin{itemize}
	\item If $0<\alpha<\alpha_0$ and $\epsilon>0$ satisfies $0<\alpha+\epsilon<1$, for $1<p<\infty$ we have
		\begin{align*}M^{\alpha+\epsilon,p}\embeds L^{\alpha,p}\embeds M^{\alpha,p}.\end{align*}
	\item If $0<\alpha<\alpha_0$ and $0<\epsilon<\alpha$ satisfies $0<\alpha+\epsilon<1$, for $1<p<\infty$ we have
		\begin{align*}B^{\alpha+\epsilon}_{p,p}\embeds L^{\alpha,p}\embeds B^{\alpha-\epsilon}_{p,p}.\end{align*}
	\item If $0<\alpha<\alpha_0$ and $\beta>0$ satisfies $\alpha<\beta<1$, for $1<p<\infty$ we have
		\begin{align*}L^{\beta,p}\embeds L^{\alpha,p}.\end{align*}
\end{itemize}

As a final result, we show that in $\RR^n$, for $\alpha<\alpha_0$, the space $L^{\alpha,p}$ coincides with the classical $\mathcal{L}^{\alpha,p}$.

Let $(S_t)_{t>0}$ be an approximation of the identity as constructed in the introduction, from a function $h$. Let $H(x)=h(|x|)$ and $H_t(x)=t^{-n}H(x/t)$. Then
\begin{itemize}
	\item $T_tf(x)=\frac{1}{t^n}\int h\left(\frac{|x-y|}{t}\right)f(y)dy=\int H_t(x-y)f(y)dy=H_t*f(x)$;
	\item $T_t1\equiv\int H_t=\int H=c_H$ for every $t>0$ and $x\in \RR^n$, then $\varphi\equiv\frac{1}{c_H}$ and $\psi\equiv 1$.
	\item $S_t f=\frac{1}{c_H^2}H_t*H_t*f=\int \left(\frac{1}{c_H^2}H_t*H_t\right)(x-y)f(y)dy$.
	\item $s(x,y,t)=\left(\frac{1}{c_H^2}H_t*H_t\right)(x-y)$.
\end{itemize}

We will see that
\begin{align*}s(x,y,t)=\varphi_t(x-y)\end{align*}
where $\varphi$ is radial. Observe
\begin{align*}H_t*H_t(x) &=\frac{1}{t^{2n}}\int H\left(\frac{x-y}{t}\right)H\left(\frac{y}{t}\right)dy=\frac{1}{t^n}\int H\left(\frac{x}{t}-z\right)H(z)dz\\
&=\frac{1}{t^n}(H*H)(x/t)=\left(H*H\right)_t(x).\end{align*}
Besides, if $\rho$ is a rotation, as $H$ is radial, we get
\begin{align*}H*H(\rho x) &=\int H(\rho x-y)H(y)dy=\int H(\rho(x-\rho^{-1}y))H(\rho\rho^{-1}y)dy\\
&=\int H(x-\rho^{-1}y)H(\rho^{-1}y)dy=H*H(x).\end{align*}

This way, if $\phi=\frac{1}{c_H^2}H*H$, we will have
\begin{align*}\frac{1}{c_H^2}H_t*H_t=\phi_t.\end{align*}

With this expression for $s$, we obtain
\begin{align*}n_\alpha(x,y) &=\int_0^\infty \alpha t^{-\alpha}s(x,y,t)\frac{dt}{t}=\int_0^\infty \alpha t^{-\alpha}\frac{1}{t^n}\phi\left(\frac{x-y}{t}\right)\frac{dt}{t}\\
&=\frac{1}{|x-y|^{n+\alpha}}\int_0^\infty \alpha u^{n+\alpha}\phi(ue_1)\frac{du}{u}=\frac{c_{n,\alpha,\phi}}{|x-y|^{n+\alpha}}\end{align*}
and the last integral converges because $\phi$ is bounded and compactly supported. 

Now, recall that for $0<\alpha<2$,
\begin{align*}\mathscr{D}_\alpha f(x)=\text{p.v. } c_{\alpha,n}\int\frac{f(y)-f(x)}{|x-y|^{n+\alpha}}dy\end{align*}
and that for those values of $\alpha$,
\begin{align*}f\in \mathcal{L}^{\alpha,p}\text{ if and only if } f,\mathscr{D}_\alpha f\in L^p.\end{align*}

From the previous result, we get
\begin{align*}D_\alpha f=C_{n,\alpha,h} \mathscr{D}_\alpha f,\end{align*}
and thus
\begin{align*}f\in \mathcal{L}^{\alpha,p}\text{ if and only if } f,D_\alpha f\in L^p.\end{align*}

In conclusion, for $0<\alpha<\alpha_0$, by the characterization theorem the spaces $L^{\alpha,p}(\RR^n)$ are independent from the choice of $h$ in the aproximation of the identity $(S_t)$, and they coincide with the classical space
\begin{align*}L^{\alpha,p}=\mathcal{L}^{\alpha,p}.\end{align*}

\section*{Acknowledgements}

The author is infinitely indebted to his advisors Eleonor `Pola' Harboure and Hugo Aimar for their guidance and support throughout the development of his doctoral thesis and its resulting papers.

\textit{E-mail address:} \texttt{mmarcos@santafe-conicet.gov.ar}

\end{document}